\documentclass[sn-mathphys,Numbered]{sn-jnl}% Math and Physical Sciences Reference Style
\usepackage{graphicx}%
\usepackage{multirow}%
\usepackage{amsmath,amssymb,amsfonts}%
\usepackage{amsthm}%
\usepackage{mathrsfs}%
\usepackage[title]{appendix}%
\usepackage{xcolor}%
\usepackage{textcomp}%
\usepackage{manyfoot}%
\usepackage{booktabs}%
\usepackage{algorithm}%
\usepackage{algorithmicx}%
\usepackage{algpseudocode}%
\usepackage{listings}%
\newcommand{\nn}{\nonumber}

\theoremstyle{thmstyleone}%
\newtheorem{theorem}{Theorem}%  meant for continuous numbers
\theoremstyle{thmstyletwo}%
\newtheorem{example}{Example}%

\theoremstyle{thmstylethree}%
\newtheorem{definition}{Definition}%

\begin{document}

\title[Double interdiction problem]{Double interdiction problem on trees on the sum of root-leaf distances by upgrading edges }


\author[1]{\fnm{Xiao} \sur{Li}}\email{alt.xiaoli@gmail.com}

\author*[1]{\fnm{Xiucui} \sur{Guan}}\email{xcguan@163.com}

\author[1]{\fnm{Junhua} \sur{Jia}}\email{230218184@seu.edu.cn}
%\author[1]{\fnm{Xinyi} %\sur{Yin}}\email{yinxinyiseu@foxmail.com}
\author[2]{\fnm{Panos M.} \sur{Pardalos}}\email{pardalos@ise.ufl.edu}

\affil*[1]{\orgdiv{School of Mathematics}, \orgname{Southeast University}, \orgaddress{\street{No. 2, Sipailou}, \city{Nanjing}, \postcode{210096}, \state{Jiangsu Province}, \country{China}}}

\affil[2]{\orgdiv{Center for Applied Optimization}, \orgname{University of Florida}, \orgaddress{\street{Weil Hall}, \city{Gainesville}, \postcode{32611}, \state{Florida}, \country{USA}}}


\abstract{
The double interdiction problem on trees (DIT) for the sum of root-leaf distances (SRD) has significant implications in diverse areas such as transportation networks, military strategies, and counter-terrorism efforts. It aims to maximize the SRD  by upgrading edge weights subject to two constraints. One gives an upper bound for the cost of upgrades under certain norm and the other specifies a lower bound for the shortest root-leaf distance (StRD). We utilize both weighted $l_\infty$ norm and Hamming distance to measure the upgrade cost and denote the corresponding (DIT) problem by (DIT$_{H\infty}$) and its minimum cost problem by (MCDIT$_{H\infty}$).  We establish the $\mathcal{NP}$-hardness of problem (DIT$_{H\infty}$) by building a reduction from the 0-1 knapsack problem.
 We solve the problem (DIT$_{H\infty}$) by two scenarios based on the number $N$ of upgrade edges. When $N=1$, a greedy algorithm with $O(n)$ complexity is proposed. For the general case, an exact dynamic programming algorithm within a pseudo-polynomial  time is proposed, which is established on a structure of left subtrees by maximizing a convex combination of the StRD and SRD. 
 Furthermore, we confirm the $\mathcal{NP}$-hardness of problem (MCDIT$_{H\infty}$) by reducing from  the 0-1 knapsack problem. To tackle problem (MCDIT$_{H\infty}$), a binary search algorithm with pseudo-polynomial  time complexity is outlined, which iteratively solves problem (DIT$_{H\infty}$). We culminate our study with numerical experiments, showcasing effectiveness of the algorithm.
}

\keywords{Network interdiction problem, upgrade critical edges, Shortest path, Sum of root-leaf distance, Dynamic programming algorithm.}



\maketitle

\section{Introduction}\label{sec1}

The landscape of terrorism has experienced marked transformation widespread use of drones. The enduring nature of this menace has precipitated continued episodes of violence in the subsequent years, compelling governments to adopt proactive strategies to forestall future calamities. A pivotal component of counterterrorism initiatives is the interception of terrorist transportation channels. Nevertheless, considering the constraints on available resources, it is essential for policymakers to prioritize and target specific transportation routes to achieve the most significant disruption. Within this paradigm, the Network Interdiction Problem (NIP), grounded in game theory, presents an insightful framework for judicious decision-making.

Generally, the NIP encompasses two main players: a leader and a follower, each propelled by distinct, often opposing, agendas. The follower seeks to maximize their goals by adeptly navigating the network to ensure the efficient transit of pivotal resources, such as supply convoys, or by augmenting the amount of material conveyed through the network. Conversely, the leader's ambition is to obstruct the follower's endeavors by strategically compromising the network's integrity.

Network Interdiction Problems that involve the deletion of edges (NIP-DE) are strategies aimed at impeding network performance by removing a set of $K$ critical edges. These strategies are pivotal in diverse fields, including transportation \cite{Network flow}, counterterrorism \cite{r2_terrorist, r7_Khachi}, and military network operations \cite{r7_Khachi}.

Significant scholarly effort has been invested in exploring NIP-DE across a spectrum of network challenges. These encompass, but are not confined to, the StRD \cite{r4_Corley_and_Sha, r5_Bar--Noy, r6_shortest path, r7_Khachi}, minimum spanning tree \cite{r11_spanning_tree, r12_spanning_tree, r13_spanning_tree, r14_spanning_tree, r15_spanning_tree}, maximum matching \cite{r16_maximum_matching, r17_maximum_matching, r18_maximum_matching, r19_maximum_matching}, maximum flow \cite{r20_maximum_flow, r21_maximum_flow}, and center location problems \cite{r23_location}.

Pioneering work by Corley and Sha in 1982 \cite{r4_Corley_and_Sha} introduced the notion of edge deletion in NIP to prolong the StRD within a network. Subsequent research by Bar-Noy in 1995 \cite{r5_Bar--Noy} established the $\mathcal{NP}$-hard nature of this problem for any $K$. Later, Khachiyan et al. (2008) \cite{r7_Khachi} demonstrated the impossibility of achieving a 2-approximation algorithm for NIP-DE. Bazgan and colleagues, in their 2015 and 2019 studies \cite{Bazgan2015, Bazgan2019}, proposed an $O(mn)$ time algorithm for incrementing path lengths by 1, and further solidified the $\mathcal{NP}$ -hard status for increments greater than or equal to 2.

Despite theoretical advances, practical application of critical edge or node deletion remains challenging. To address these practical limitations, Zhang et al. (2021) \cite{Zhang_SPIT} introduced an upgraded framework for NIP that focuses on edge upgrades. They explored this concept through the Shortest Path Interdiction Problem (SPIT) and its variant, the Minimum Cost SPIT (MCSPIT), on tree graphs. For SPIT, an $O(n^2)$ primal-dual algorithm was provided under the weighted $l_1$ norm, with the complexity improved to $O(n)$ for the unit $l_1$ norm. They extended their investigation to unit Hamming distance, designing algorithms with complexities $O(N + l \log l)$ and $O(n(\log n + K^3))$ for $K=1$ and $K>1$, respectively \cite{Zhang_MSPH}. Subsequently, Lei et al. (2023) \cite{Lei} enhanced these to $O(n)$ and $O(nK^2)$ time complexities.

In a recent study, Li et al. (2023) \cite{Xiao} addressed the sum of root-leaf distances (SRD) interdiction problem on trees with cardinality constraint by upgrading edges (SDIPTC), and its related minimum cost problem (MCSDIPTC). Utilizing the weighted $l_\infty$ norm and the weighted bottleneck Hamming distance, they proposed two binary search algorithms with both time complexities of $O(n \log n)$ for problems (SDIPTC) and two binary search algorithms within O($N n^2$) and O($n \log n$) for problems (MCSDIPTC), respectively.  However, these introductive problems did not limit the shortest root-leaf  distance (StRD), which makes the upgrade scheme lack of comprehensiveness and rationality. To remedy this, we introduce the double interdiction problem on the sum of root-leaf distance by upgrading edges that have restriction both on the SRD and StRD.

Specifically, certain advanced transportation networks can be visualized as rooted trees.\cite{Zhang_SPIT} In this model, the root node denotes the primary warehouse, the child nodes signify intermediary transit points, and the leaf nodes portray the ultimate users or terminals. Serving as the leader in this scenario, our aim is to proficiently impede and neutralize this network. The corresponding challenges are articulated as follows.

Let $T=(V,E,w,u,c)$ be an edge-weighted tree rooted at $s$, where $V:=\left \{ s,v_1,v_2,\dots,v_n \right \} $ and
$E:=\left \{ e_1,e_2,\dots,e_n \right \}$ are the sets of nodes and edges, respectively.
Let $Y:=\left \{ t_1,t_2,\dots,t_m \right \}$ be the set of leaf nodes and $S(v):=\{ v'|v'  \text{ is the son of } v \}$.
Let $w(e)$ and $u(e)$ be the original weight and upper bound of upgraded weight for
edge $e\in E$, respectively,
where $u(e)\geq w(e) \geq 0$.
Let $c(e)>0, r(e) \in \mathbb{Z}
 >0$ be the unit modification cost of edge $e\in E$ for $l_\infty$ norm and weighted sum Hamming distance, respectively. Let $P_k:=P_{t_k}:=P_{s,t_k}$ be the root-leaf path from the root node $s$ to the leaf node $t_k$.
Let $w(P_k):=\sum_{e\in P_{k}}w(e)$ and $w(T):=\sum_{t_k\in Y}w(P_k)$ be the weight of the path $P_k$ and SRD under the edge weight $w$, respectively. Let $In_{\bar w}(e)$ represent the increment of SRD by edge $e$ under edge weight vector $\bar w.$

Given two values $K$ and $M$, the double interdiction problem on trees on the sum of root-leaf distance by upgrading edges  (\textbf{DIT}) aims to maximize SRD by determining an edge weight vector $\tilde{w}$ such that the modification cost $\|\tilde{w}-w\|$ in a certain norm does not exceed $K$, and the StRD from root $s$ to any leaf node must not be less than $M$. The mathematical representation of this problem can be articulated as follows.
\begin{eqnarray}
	&\max \limits_{\tilde{w}}& \tilde{w}(T):=\sum_{t_k \in Y}\tilde{w}(P_k)\nonumber\\
       \textbf { (DIT)} & s.t. & \min \limits_{t_k \in Y} \tilde{w}(P_k)\ge M  \label{DITH}\\
	 &&\ \| \tilde{w}-w\|  \leq K,\nonumber\\
	&& w(e) \leq \tilde{w}(e) \leq u(e), \quad e \in E.\nonumber
\end{eqnarray}

Its related minimum cost problem (MCDIT) by exchanging its objective function and the modification cost can be formulated as follows.

\begin{eqnarray}
	&\min \limits_{\tilde{w}} & C(\tilde{w}):=\ \| \tilde{w}-w\| \nonumber  \\
       \textbf {(MCDIT) }  &s.t.& \min \limits_{t_k \in Y} \tilde{w}(P_k)\ge M \label{MCDITH} \\
	 && \tilde{w}(T) \geq D, \nonumber\\
	&&w(e) \leq \tilde{w}(e) \leq u(e), \quad e \in E\nonumber.
\end{eqnarray}
In practical applications, several norms are employed to quantify the cost of modifications, notably including the \(l_1\) norm, \(l_\infty\) norm, bottleneck Hamming distance, and weighted sum Hamming distance. Each of these norms finds extensive use in various domains. For instance, the \(l_\infty\) norm is instrumental in ensuring that traffic on a single network link does not surpass its maximum capacity at any given time, thereby preventing congestion and optimizing network traffic distribution issues \cite{Network flow}. Similarly, the weighted sum Hamming distance is applied to manage the number of interfaces within an optical network, as discussed in \cite{Optical networks}. Under certain specific conditions, both the \(l_\infty\) norm and weighted sum Hamming distance are utilized to assess costs, particularly when modifications involve more than one resource. However, existing research has overlooked the scenario involving the simultaneous application of both \(l_\infty\) and weighted sum Hamming distances. This paper aims to address this gap by measure the upgrade cost with both \(l_\infty\) and weighted sum Hamming distances, as described below.
\begin{eqnarray}
	&\max \limits_{\tilde{w}} &\tilde{w}(T):= \sum_{t_k \in Y}{\tilde{w}}(P_k)\nonumber\\
	\textbf{(DIT$_{H\infty}$)} & s.t. &  \min \limits_{t_k \in Y} \tilde{w}(P_k)\ge M, \label{DITH_Model} \\ 
         && \max_{\ e\in E}c(e)(\tilde{w}(e)-w(e))\le K, \nonumber\\
	&&\sum_{e\in E} r(e)H(\tilde{w}(e),w(e))\le N, \nonumber\\ 
	&& w(e) \leq \tilde{w}(e) \leq u(e), \quad e \in E.\nonumber
\end{eqnarray}
where $H(\tilde{w}(e),w(e))=\begin{cases}
	0, &{\tilde{w}}(e)=w(e)\\
	1,&{\tilde{w}}(e)\ne w(e)\\
\end{cases}\nonumber 
$ is the Hamming distance between 
$\tilde{w}(e)$ and $w(e)$ and $N$ is a given     positive value.

Its related minimum cost problem obtained by exchanging the $l_\infty$ norm cost and the SRD objective function can be written as 
\begin{eqnarray}
    &\min \limits_{\tilde{w}}& C(\tilde{w}) := \max_{e\in E}c(e) (\tilde w(e) - w(e)) \nonumber  \\
    (\mathbf{MCDIT}_{H\infty})  &\text{s.t.}& \min \limits_{t_k \in Y} \tilde{w}(P_k) \ge M, \label{MCDITH_inf}\\
        && \sum_{t_k \in Y} \tilde{w}(P_k) \geq D, \nonumber\\
        && \sum_{e\in E} r(e)H(\tilde{w}(e), w(e)) \le N, 
 \nonumber \\
        && w(e) \leq \tilde{w}(e) \leq u(e), \quad e \in E. \nonumber
\end{eqnarray}

The structure of the paper is organized as follows: Section 2 establishes the $\mathcal{NP}$-hardness of the problem (DIT$_{H\infty}$) by demonstrating a reduction from the 0-1 knapsack problem. In Section 3, we introduce a dynamic programming algorithm \ref{Alg_PIC_inf} to solve the problem (DIT$_{H\infty}$), albeit with pseudo-polynomial  time complexity. Moving to Section 4, the paper delves into proving the $\mathcal{NP}$-hardness of the minimum cost problem (MCDIT$_{H\infty}$) through a two-step process. In Section 5, we address the problem (MCDIT$_{H\infty}$) by employing a binary search algorithm which iteratively calling Algorithm \ref{Alg_PIC_inf}. Section 6 is dedicated to presenting the outcomes of computational experiments, which affirm the efficiency and accuracy of the proposed algorithms. The paper concludes in Section 7, where we summarize the key findings and outline potential directions for future research in this domain.

\section{The $\mathcal{NP}$-Hardness of problem (DIT$_{H\infty}$)}

When the weighted $l_\infty$ norm and weighted sum Hamming distance  is applied to the upgrade cost, the problem (DIT$_{H\infty}$) is formulated in (\ref{DITH_Model}). Note that the weighted sum Hamming distance is discrete, posing challenges in its treatment. To gain a clearer understanding of problem (DIT$_{H\infty}$), we initially examine its relaxation (DIT$_\infty$) by removing the constraint of weighted sum Hamming distance. Its mathematical model can be outlined as follows.
\begin{eqnarray}
	&\max \limits_{\tilde{w}} &\tilde{w}(T):= \sum_{t_k \in Y}{\tilde{w}}(P_k)\nonumber\\
	\textbf{(DIT$_\infty$)} & s.t. & \min \limits_{t_k \in Y} \tilde{w}(P_k)\ge M,  \label{DIT_Model} \\ 
         && \max_{\ e\in E}c(e)(\tilde{w}(e)-w(e))\le K, \nonumber\\
	&& w(e) \leq \tilde{w}(e) \leq u(e), \quad e \in E.\nonumber
\end{eqnarray}

In the problem (\textbf{DIT$_\infty$}), we can maximize the weight of each edge to the greatest extent possible under the constraint of cost $K$ and the upper bound $u(e)$  as follows

\begin{equation}
    \bar w(e)=\min \left   \{ w(e)+ \frac{K}{c(e)},u(e) \right \}(e\in E) .\label{DIT_inf_ans} 
\end{equation}

If, in this scenario, the weight of the StRD remains less than M, expressed as $\min \limits_{t_k \in Y} \bar{w}(P_k) < M$, it implies problem is infeasible with the following theorem.
\begin{theorem} \label{DIT_ans_theorem}
Let $\bar{w}$ be defined as shown in Equation (\ref{DIT_inf_ans}). If $\min \limits_{t_k \in Y} \bar{w}(P_k) < M$, then the problem (DIT$_\infty$) is infeasible. Otherwise, the solution $\bar w$  by Equation (\ref{DIT_inf_ans}) is optimal for the problem (DIT$_\infty$) and thus can be obtained in $O(n)$ time.
\end{theorem}

\begin{proof}
When $\min_{t_k \in Y} \bar{w}(P_k) < M$, and all edges have been adjusted to their maximum permissible values, it indicates that the StRD has reached its highest potential length. If, in such a scenario, the StRD still does not fulfill the specified constraints, it indicates that the problem cannot be solved. Conversely, if the StRD meets the constraints under these conditions, it signals the presence of a viable solution. And note that all edges are already at their upper limits, precluding any possibility for further enhancement, the maximal SRD is also achieved. For Equation (\ref{DIT_inf_ans}), we just upgrade every edges, which takes $O(n)$ time.
\end{proof}

To represent the SRD increment by upgrading one edge, we introduce the following definition.
\vspace{0.5em}
\begin{definition} \cite{Xiao}
	For any $ e \in E$, 
	define $L(e):=\left \{ t_{k}\in Y \vert e \in P_{s,t_k}\right \}$
	as the set of leaves $t_{k}$ to which $P_{s,t_k}$ passes through $e$.
	If $t_k \in L(e)$, then $t_k$ is controlled by the edge $e$. 
        Similarly, for any $v \in V$, define $L(v):=\left \{ t_{k}\in Y \vert v \in P_{s,t_k} \right \}$ as the set of leaves  controlled by the node $v$.
\end{definition}
\vspace{0.5em}

Using this definition, we know that the increment of SRD can be expressed as
\begin{equation} \label{equ_In(e)}
    In_{\bar w}(e):= |L(e)|(\bar w(e) -w(e)).
\end{equation}

Building upon the discussion before, we understand that under weighted $l_\infty$ norm, 
it is advantageous to upgrade an edge to its upper bound determined by Equation (\ref{DIT_inf_ans}). Consequently, the problem can be reformulated into a new 0-1 Integer Linear Programming model, where the binary decision variable $x(e)$ is assigned to each edge $e \in E$. An edge $e$ is considered upgraded or not if $x(e)=1$ or $0$.
 Thus, the original problem (\ref{DITH_Model})  can be transformed into the following general form.

\begin{subequations}\label{GDITH_Program}
\begin{align}
    & \max \; \sum_{e \in E} In_{\bar w}(e)x(e) & & \notag \\
   \text{\textbf{(GDIT$_{H\infty}$)}}     &  \; \text{s.t.}  \sum_{e \in P_k } (\bar{w}(e)-w(e)) x(e) \geq M -w(P_k), \forall k=1,\dots ,m. & & \label{GDITH_Program_a} \\
   &   \; \quad \; \; \sum_{e \in E} r(e)x(e) \leq N, & &  \label{GDITH_Program_b}\\
    & \quad \quad x(e)  \in \{ 0,1\}, e \in E. & & \notag 
\end{align}
\end{subequations}

Here, the objective function is transformed from $\tilde w(T)$ in problem (\ref{DITH_Model}) to the increment of SRD by upgrading edges as $\tilde w(T)-w(T)=\sum_{e \in E} In_{\tilde w}(e)=\sum_{e \in E} In_{\bar w}(e)x(e)$. Meanwhile, note that $\min_{t_k \in Y} \tilde{w}(P_k) \ge M$ in problem (\ref{DITH_Model}) is equivalent to $\tilde{w}(P_k) \ge M$ for any $t_k \in Y$, then the StRD constraint can be interpreted as (\ref{GDITH_Program_a}). Consequently, we arrive at the following theorem.
\begin{theorem} \label{GDITH_is_equal_to_DITH}
Problem \textbf{(GDIT\(_{H\infty}\))} is equivalent to the problem \textbf{(DIT\(_{H\infty}\))}.
\end{theorem}

\begin{proof}
For clarity, let us denote problem \textbf{(GDIT\(_{H\infty}\))} as \(Q_1\) and problem \textbf{(DIT\(_{H\infty}\))} as \(Q_2\). The theorem is established through the validation of the following two statements.

(i) For every feasible solution $x_1$ to \(Q_1\), there exists a feasible solution $w_2$ to \(Q_2\), with equal or higher objective value. By Theorem \ref{DIT_ans_theorem}, we can upgrade every edge $e$ with $x_1(e)=1$ to the upper bound $\bar w(e)$ defined by Equation (\ref{DIT_inf_ans}) to obtain a new solution $w_2$ of $Q_2$ defined below
\[
w_{2}(e)=\begin{cases}
\bar{w}(e), &  x_1(e)=1, \\
w(e), & \text{otherwise}.
\end{cases}
\]
Obviously, \(w_{2}\) satisfies all the constraints within \(Q_2\) and have the same SRD value. 

(ii) For every feasible solution $w_2$ to \(Q_2\), there exists a feasible solution $x_1$ to \(Q_1\) defined below
\[
x_{1}(e)=\begin{cases}
1, &  w_2(e) \ne w(e) , \\
0, & \text{otherwise}.
\end{cases}
\]
Then \(x_1\) have exactly the same objective value as \(w_2\), and (ii) is established.
\end{proof}

Using Theorem \ref{GDITH_is_equal_to_DITH}, we can show the $\mathcal{NP}$-hardness of problem (DIT$_{H\infty}$) by the following theorem.

\begin{theorem} \label{thm-DIT-NPH}
The problem \textnormal{(DIT$_{H\infty}$)} is \textnormal{$\mathcal{NP}$}-hard.
\end{theorem}

\begin{proof}
To establish the \textnormal{$\mathcal{NP}$}-hardness of Problem \textnormal{(DIT$_{H\infty}$)}, we reduce from the well-known \textnormal{$\mathcal{NP}$}-hard \textnormal{0-1} knapsack problem, which is defined as follows:
\begin{eqnarray}
&\max \quad & \sum_{i=1}^{n} p_{i} x_{i} \nonumber\\
&\text{s.t.} \quad & \sum_{i=1}^{n} c_{i} x_{i} \le R, \label{eq:01knapsack}\\
&\quad & x_{i} \in \{0,1\}, i = 1, \ldots, n. \nonumber
\end{eqnarray}
where each $c_i$ is a positive integer and $R$ is a constant.

Given any instance $I_1$ of the 0-1 knapsack problem in \eqref{eq:01knapsack}, we construct an instance of problem (GDIT$_{H\infty}$) in (\ref{GDITH_Program}) as follows. Let $In_{\bar{w}}(e_i) := p_i, r(e_i) := c_i, N:=R$ and $M := 0$.
%, we obtain $\bar{w} = u$ and the shortest path constraint is trivially satisfied. Since $|L(e_i)| = 1$ for all edges $e_i$ on a chain, it follows that
%\[In_{\bar{w}}(e_i) = |L(e_i)| \cdot (\bar{w}(e_i) - w(e_i)) = u(e_i) - w(e_i) = p_i,\]
Furthermore, $x_i=1$ in an instance $I_1$ if and only if $x(e_i)=1$ in an instance of (GDIT$_{H\infty}$). 
%demonstrating that the transformed instance of the (GDIT$_{H\infty}$) problem is exactly instance $I_1$. 
This equivalence shows that problem \textnormal{(GDIT$_{H\infty}$)} is \textnormal{$\mathcal{NP}$}-hard, thereby establishing the \textnormal{$\mathcal{NP}$}-hardness of problem \textnormal{(DIT$_{H\infty}$)} by Theorem \ref{GDITH_is_equal_to_DITH}.
\end{proof}

\section{Solve the problem \textbf{(DIT\(_{H\infty}\))} }
When applying the $l_\infty$ norm and the weighted sum Hamming distance to the double interdiction problem, the complexity of the situation escalates significantly. This is due to the necessity to identify crucial edges that not only comply with the StRD constraint but also maximize the SRD value. To navigate the challenges presented by the integration of weighted sum Hamming distance, we introduce an enhanced dynamic programming approach. Our methodology begins with an exhaustive case-by-case examination. In particular, the scenario becomes markedly simpler when upgrading a single edge is permitted, serving as the initial case for our analysis. As the complexity of the scenarios increases with the general case, our focus shifts to formulating a dynamic programming objective function, denoted as $h$. This function employs a convex combination to simultaneously cater to the augmentation in SRD and the necessity to minimize the path length. Following this, we define a transition equation based on the structure of left subtrees, facilitating the execution of a dynamic programming iteration. Moreover, we implement the binary search technique iteratively to fine-tune the optimal parameters, ultimately leading to the determination of the optimal solution to the original problem.

\subsection{Solve problem (DIT$_{H\infty}$) when upgrading one edge}
We denote the problem (DIT$_{H\infty}$) by (DIT$_{H\infty}^1$) when $N=1$ and $r(e)=1$ for all $e \in E$, which aims to modify a single critical edge to maximize SRD while ensuring compliance with the StRD constraint. To approach this task, we initially set forth some necessary definitions. Subsequently, we employ a greedy algorithm to efficiently address the problem, achieving a solution in \(O(n)\) time complexity. 

%\vspace{0.5em}

\begin{definition}
	For any $e^* \in E$, we define the upgrade vector $w_{e^*}$ as follows.
\begin{equation}
    w_{e^*}(e)=
    \begin{cases}
        \bar{w}(e), & \text{if } e =e^* \\
        w(e), & \text{otherwise.}
    \end{cases}    
\end{equation} \nonumber
\end{definition}

In order to determine our next optimization goal, we need to consider whether the constraints required in the original situation are satisfied.
\begin{theorem}
    If $\min_{t_k \in Y}w(P_k) \ge M$, let $e^*:= \operatorname{arg\,max}_{e \in E} In_{w_e}(e)$, the optimal solution of the problem (DITH$_\infty^1$) is $w_{e^*}$.
\end{theorem}

\begin{proof}
    When $\min_{t_k \in Y}w(P_k) \ge M$, the StRD constraint is trivally satisfied. Our objective then shifts to maximizing the SRD by upgrading one edge. Naturally, the optimal solution selects the edge $e^*$ with the greatest SRD increment.
\end{proof}

\begin{theorem} \label{thm_shortest}
     If $\min_{t_k \in Y}w(P_k) < M$, let $k^*:= \operatorname{arg\,min}_{t_k \in Y} w(P_k)$, then there must be an optimal solution $w_{e^*}$ for some $e^* \in P_{t_{k^*}}$.
\end{theorem}
\begin{proof}
The necessity of the upgrade edge being a part of the shortest path is evident for it to have an impact in the constraint.
\end{proof}

Consequently, it has been established that the target edge exists within the shortest path $P_{k^*}$ connecting the source node $s$ to the leaf node $t_{k^*}$. In order to further decompose the problem, we consider the following cases.

\textbf{Case 1.} When $\max_{e \in P_{k^*}} (\bar w(e) -w(e)) <M-w(P_{k^*}).$ In this case, the problem is infeasible since upgrading any edge would not satisfy the constraint.

\textbf{Case 2.} When $\max_{e \in P_{k^*}} (\bar w(e) -w(e)) \ge M-w(P_{k^*})$. Let $Y_1:=\{t_k\in Y| w(P_{k}) < M\}$ be the set of leaves not satisfying (\ref{GDITH_Program_a}). Let $B := \{ e \in P_{k^*} |(\bar w(e) -w(e)) \ge M-w(P_{k^*}), Y_1 \subseteq L(e) \}$.

\textbf{Case 2.1} $|Y_1|=1$.

We need to upgrade $\tilde e$ that satisfies $\bar w(\tilde e) -w(\tilde e) \ge M-w(P_{k^*})$ and also has the largest SRD increment. To be specific, let $\tilde e:= \operatorname{arg\,max}_{ e \in 
B} In_{w_e}(e)$, and then  $\tilde w:=w_{\tilde e}$ is the optimal solution.

\textbf{Case 2.2.} $|Y_1|> 1$ and $B=\emptyset$.
%When $\max_{e \in P_{k^*}} (\bar w(e) -w(e)) \ge M-w(P_{k^*}).$ And there are more than one path that don't satisfy the constraint, let the leaves denoted by $Y_1:=\{ t_{\alpha_1},t_{\alpha_2},\dots \}$, The total situation can be presented as the following two distinct scenarios.

%\textbf{Scenario 1:} If $|Y_1|\ge 1$ and the set $B := \{ e \in E |(\bar w(e) -w(e)) \ge M-w(P_{k^*}), Y_1 \subseteq L(e) \}$ is empty, 
In this case, the problem is infeasible. Modifying a single edge to extend the StRD is ineffective, as this alteration fails to impact all existing shortest paths.

\textbf{Case 2.3.} $|Y_1|> 1$ and $B\neq\emptyset$. 

In this case, upgrade any edge in $B$ can satisfy all StRD constraint and therefore we choose the one $e^*:= \operatorname{arg\,max}_{e \in B} In_{w_e}(e)$ with the largest SRD increment. Then the optimal solution is identified as  $\tilde{w}=w_{e^*}$. 

By the above analyze, we give the following theorem.
\begin{theorem}
    The $\tilde{w}$ defined in \textbf{Case 2.1} and \textbf{Case 2.3.} is an optimal solution of problem (DIT$_{H\infty}^1$).
\end{theorem}
\begin{proof}

For \textbf{Case 2.1}, suppose there exists a superior solution \( w' \) with larger SRD,  let the upgraded edge be denoted by \( e' \). According to Theorem \ref{thm_shortest}, we have \( e' \in P_{k^*} \), and by the definition of $\tilde w$,  it must hold that $ (\bar w(e) -w(e)) < M-w(P_{k^*})$. This leads to a contradiction with feasibility.

For \textbf{Case 2.3}, assume a superior solution \( w' \) exists with larger SRD than $\tilde w$,  and denote the upgraded arc by \( e' \). Then, form Scenario 1 we know that $e' \in B$, by definition, $\tilde w$ should have the largest SRD, which is a contradiction.
\end{proof}

Therefore, we have the following algorithm \ref{Alg_DITH_inf_N1} to solve (DIT$_{H\infty}^1$).

\begin{algorithm}
	\caption{A greedy algorithm to solve (DIT$_{H\infty}^1$).}
 \label{Alg_DITH_inf_N1}
	\small
	\begin{algorithmic}[1]
		\Require A tree $T$ rooted at $s$, two edge weight vectors $w$ and $u$, two values $K$ and $M$.
		\Ensure The optimal value $w_{e^*}(T)$, the upgraded edge $e^*$ and the optimal solution $w_{e^*}$.
            \State Calculate $\bar w$ as in (\ref{DIT_inf_ans}). Let $k^*:= \operatorname{arg\,min}_{t_k \in Y} w(P_k)$.
            \If{$\min_{t_k \in Y}w(P_k) \ge M$}
            \State Let $e^*:= \operatorname{arg\,max}_{e \in E} In_{w_e}(e)$.
            \State \Return $w_{e^*}(T), e^*, w_{e^*}$.
            \ElsIf{$\max_{e \in P_{k^*}} (\bar w(e) -w(e)) \ge M-w(P_{k^*})$}        
                       
            \State Let $Y_1:=\{t_k\in Y| w(P_{k}) < M\}$.
            \State Let  $B := \{ e \in P_{k^*} |(\bar w(e) -w(e)) \ge M-w(P_{k^*}), Y_1 \subseteq L(e) \}$.
            \If{$|Y_1|=1$}
            \State Let $e^*:= \operatorname{arg\,max}_{  e \in  
B } In_{w_e}(e)$.
            \State \Return $w_{e^*}(T),e^*, w_{e^*}.$
            \Else           
                      
            \If{$B\neq\emptyset$}
            \State Let $e^*:= \operatorname{arg\,max}_{e \in B} In_{w_e}(e)$
            \State \Return $w_{e^*}(T),e^*, w_{e^*}.$
            \Else
            \State \Return No feasible solution.   
            \EndIf
        \EndIf
        \Else 
        \State \Return No feasible solution.
      
      \EndIf
	\end{algorithmic}
\end{algorithm}
Given that finding the maximum value in a set takes $O(n)$ time, we can draw the following conclusion.
\begin{theorem}
The problem (DIT$_{H\infty}^1$)  can be solved by Algorithm \ref{Alg_DITH_inf_N1} in O($n$) time.
\end{theorem}

\subsection{Solve the general problem (DIT$_{H\infty}$)}
For the general form of problem (DIT$_{H\infty}$), things escalates in complexity. When deciding whether to upgrade an edge, a delicate balance between the StRD and SRD must be struck. We begin by laying down foundational definitions. Subsequently, we formulate an objective function that accounts for both factors which results in a Combination Interdiction problem on Trees ($CIT_{H \infty}^{\lambda}$). Constructing the state transition equation marks the completion of one iteration. Ultimately, we iterate to establish optimal parameters and uncover the optimal solution, a process that incurs pseudo-polynomial  time complexity.
%\vspace{0.5em}
\begin{figure}[htbp!!!!]
    \centering
    \includegraphics[scale=0.25]{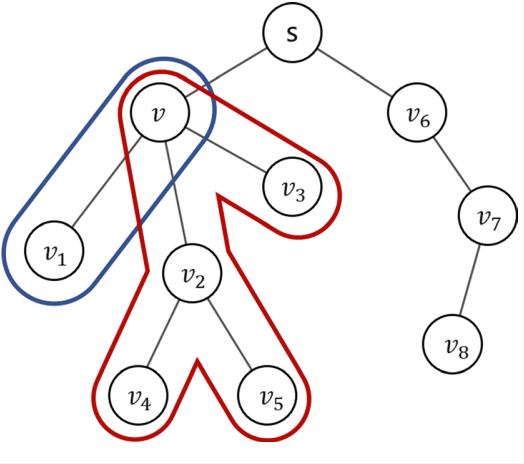}
    \caption{The area labeled in blue 
and red are the subtrees $T_v(1)$
and $T_v(2:3)$, respectively \cite{Lei}.}
    \label{fig:subtree}
\end{figure}
\begin{definition} \cite{Zhang_MSPH}
    Define  $T_{v}=\left(V_{v}, E_{v}\right)$  as the subtree rooted at  $v$  for any  $v \in V$. Let  $S(v)=\left\{v_{s_{1}}, \ldots, v_{s_{|S(v)|}}\right\}$ be the son set, where $v_{s_{q}}$ the $q$-th son  for any non-leaf node  $v$. Let $P(v_1,v_2)$ the unique path from $v_1$ to $v_2$. Define the left  $q$-subtree of  $v$  as  $T_{v}(q)=T_{v_{s_{q}}} \cup P(v, v_{s_{q}})$  and the  $p:q$-subtree of  $v$  as  $T_{v}(p:q)=\left(\bigcup_{i=p}^{q} T_{v}(i)\right) \cup\{v\}$. Specially, define  $T_{v}(p: 0)=\emptyset$.
\end{definition}

\subsubsection{A dynamic programming algorithm to solve ($CIT_{H \infty}^{\lambda}$)}

To achieve a balance between the StRD and SRD, we introduce the Combination Interdiction problem on Trees ($CIT_{H \infty}^{\lambda}$), which employs a convex combination of these two factors and $0\le \lambda \le 1$ is a parameter. This problem is specifically applied to the structure of the left $p:q$-subtree. Upon solving ($CIT_{H \infty}^{\lambda}$), we can demonstrate the existence of an optimal parameter  $\lambda^*$ such that the solution to ($CIT_{H \infty}^{\lambda}$) perfectly aligns with the solution to the general problem ($DIT_{H\infty}$). Suppose $T_v(p:q)=(V',E')$.
\begin{eqnarray}
&\max \limits_{\hat{w}}& \left\{ (1-\lambda) \sum_{e \in E'}|L(e)|\hat{w}(e) + \lambda \min \limits_{t_k \in \cup_{i=p}^{q}L(v_{_i}) }   \hat{w}(P(v,t_k))\right\} \nonumber\\
(\textbf{CIT}_{H\infty}^{\lambda}) 
&s.t.&   \sum_{e\in E'}r(e)H(\hat{w}(e),w(e))\le N,\label{PIC_inf} \\ \nonumber
&& w(e) \leq \hat{w}(e) \leq \bar{w}(e), \quad e \in E'.\nonumber
\end{eqnarray}
\begin{definition}
     Define $h(v,p:q,N,\lambda)$ the optimal value and  $\hat w_{(v,p:q,N,\lambda)}$ an optimal solution to the problem ($CIT_{H \infty}^{\lambda}$). 
Specifically, define $h(v,p:q,N,\lambda)=+\infty$, if $p>q$ or $N<0$. And define the StRD and SRD under $h(v,p:q,N,\lambda)$ as $SP_{(v,p:q,N,\lambda)}= \min _{t_k \in \cup_{i=p}^{q}L(v_{s_i}) }  \hat{w}(P(v,t_k))$ and  $SRD_{{(v,p:q,N,\lambda)}}=\sum _{e\in E'}|L(e)|\hat{w}(e)$, respectively.
\end{definition}

 Through these definitions, we can perform dynamic programming methods to solve problem ($CIT_{H \infty}^{\lambda}$). We first show the state transition from tree $T_v(p,p)$ to its subtree $T_{v_{s_p}}(1:|S(v_{s_p})|)$.
\begin{theorem} \label{thm_cal_1}
    For any non-leaf node $v$ and any integer $k \ge 1$, let $e_{s_p}:= (v,v_{s_p})$, then we have
    \begin{equation} \label{DP:EQ1}
    \begin{array}{ll}
         &h(v,p:p,k,\lambda)= \max \{ h\left(v_{s_p}, 1:|S(v_{s_p})|, k,\lambda\right)+\lambda w(e_{s_p})+ (1-\lambda) In_w(e_{s_p}),  \\
         &  h\left(v_{s_p}, 1:|S(v_{s_p})|, k-r(e_{s_p}),\lambda \right)+\lambda \bar{w}(e_{s_p})+ (1-\lambda) In_{\bar w}(e_{s_p}) \}  
    \end{array}
    \end{equation}

\end{theorem}
\begin{proof}
Suppose $\hat w_p$  is an optimal solution corresponding to the objective value $h(v,p:p,k,\lambda)$, then there are the following two cases for $\hat w_p$ depending on whether the edge $e_{s_p}$ is upgraded or not.

\textbf{Case 1:} $H(\hat w_p(e_{s_p}),w(e_{s_p}))=0$, which means the edge $e_{s_p}$ is not upgraded and $\hat w_p(e_{s_p})=w(e_{s_p})$, then  $SP_{(v,p:p,k,\lambda)}=w(v_{s_p})+ SP_{(v_{s_p},1:|S(v_{s_p})|,k, \lambda)}$.  So the StRD increases $w(e_{s_p})$ by edge $e_{s_p}$, and SRD increases by $In_{ w}(e_{s_p})$, which renders the first item in \eqref{DP:EQ1}.

\textbf{Case 2:} $H(\hat w_p(e_{s_p}),w(e_{s_p}))=1$. In this case, the edge $e_{s_p}$ is upgraded to $\bar w(e_{s_p})$. Then $SP_{(v,p:p,k,\lambda)}=\bar w(e_{s_p})+SP_{(v_{s_p},1:|S(v_{s_p})|,k-r(e_{s_p}),\lambda)}$. Hence StRD increases $\bar w(e_{s_p})$, SRD increases by $In_{\bar w}(e_{s_p})$, which renders the second item in \eqref{DP:EQ1}.
\end{proof}

Next, we show that the problem ($CIT_{H \infty}^{\lambda}$) defined on $T_v(p:q)$ can be divided into two sub-trees $T_v(p:l)$ and $T_v(l+1:q)$ with the following theorem.

\begin{theorem}\label{thm_cal_2}
Suppose that v is a non-leaf node. For any child node index $p,l,q$ satisfying $p \le l< r$, then the optimal value of the problem ($CIT_{H \infty}^{\lambda}$) defined on $T_v(p:q)$ can be calculated by
\begin{eqnarray}   
h(v,p:q,k,\lambda)=\max_{ k_1+k_2 \le k} &\{ & \lambda   \min\{
SP_{(v,p:l,k_1,\lambda)},SP_{(v,l+1:q,k_2,\lambda)} \}\nonumber\\
&+& (1-\lambda)(SRD_{(v,p:l,k_1,\lambda)}+SRD_{(v,l+1:q,k_2,\lambda)}) \}.\label{eq:hvpp}
\end{eqnarray}
\end{theorem}

\begin{proof}
Let $E_1$, $E_2$ and $E_3=E_1 \cup E_2$ be the edge sets of $T_v(p:l)$, $T_v(l+1:q)$, and  $T_v(p:q)$, respectively.
Suppose $w^*:=\hat w_{(v,p:q,k,\lambda)}$ is the optimal solution, let $k_1^*:=\sum_{ e\in E_1,w^*(e)\ne w(e) } r(e)$, $k_2^*:=\sum_{ e\in E_2,w^*(e)\ne w(e) } r(e)$, then $k_1^*$ and $k_2^*$ are both integers and $k_1^*+k_2^* \le k$.
\begin{eqnarray}
&&h(v,p:q,k,\lambda) \nn\\
&=&\min \limits_{t_k \in \cup_{i=p}^{q}L(v_{s_i}) }  \lambda w^*(P(v,t_k))+(1-\lambda) \sum_{e \in E_3} |L(e)| w^*(e) \nn\\
&=&\lambda\min \left\{\min _{t \in \bigcup_{i=p}^{l} L(v_{s_i})}  w^*(P(v, t)), \min _{t \in \bigcup_{i=l+1}^{q} L(v_{s_i})} w^*(P(v, t))\right\}\nn\\
&&\quad +(1-\lambda)\sum_{e \in E_1}|L(e)|w^*(e)+(1-\lambda)\sum_{e \in E_2}|L(e)|w^*(e)\nn\\
& =& \lambda \min \left\{SP_{ \left(v, p: l, k_{1}^*, \lambda\right)}, SP_{ \left(v, l+1: q, k_{2}^*,\lambda \right)} \right\}\nn\\
&&\quad +(1-\lambda)(SRD_{(v,p:l,k_1^*,\lambda)}+SRD_{(v,l+1:q,k_2^*,\lambda)}).\nn
\end{eqnarray}
The last equation comes from the definition of problem ($CIT_{H \infty}^{\lambda}$), which reveals a good property that the optimal value $h(v,p:q,k,\lambda)$  on $T_v(p:q)$ comes from two optimal solutions $\hat{w}_{(v,p:l,k_1^*,\lambda)}$ and $\hat{w}_{(v,l+1:q,k_2^*,\lambda)}$ on two sub-trees $T_v(p:l)$ and $T_v(l+1:q)$.
Therefore the theorem holds.
\end{proof}

Using the two theorems above, we are able to calculate function $h$ in this two step.
\begin{equation} \label{DITH_2eq_a}
\begin{array}{r}
h(v, 1: p, N,\lambda)=\max\limits_{0\le N_1 \le N, N_1 \in \mathcal{Z}^+} \{ \lambda \min\{SP_{ (v,1:p-1,N_1,\lambda}),SP_{ (v,p:p,N-N_1, \lambda) }\}+ \\
(1-\lambda) (SRD_{(v,1:p-1,N_1,\lambda)}+SRD_{(v,p:p,N-N_1, \lambda)}) \}\\
\end{array}
\end{equation}
\begin{eqnarray} \label{DITH_2eq_b}
&&h(v,p:p,N-N_1,\lambda)\nn\\
&=&\max\{ h\left(v_{s_{p}}, 1:|S(v_{s_{p}})|, N-N_1, \lambda\right)+\lambda w(e_{s_p})+(1-\lambda) |L(e_{s_p})| w(e_{s_p}),\\
&&h\left(v_{s_{p}}, 1:|S(v_{s_{p}})|, N-N_1-r(e_{s_p}), \lambda\right)+
\lambda \bar{w}(e_{s_p})+(1-\lambda) |L(e_{s_p})| \bar{w}(e_{s_p}) \}.\nn
\end{eqnarray}

Throughout the discussion above, we propose Algorithm \ref{Alg_PIC_inf} to solve the problem ($CIT_{H \infty}^{\lambda}$), where we need to call Depth-First Search algorithm 
\textbf{DFS}($s,1:|S(s)|, N, \lambda)$ recursively \cite{Intro_ALG}.
\begin{algorithm}
	\caption{A dynamic programming algorithm to solve  ($CIT_{H \infty}^{\lambda}$).}
 \label{Alg_PIC_inf}
	\small 	\begin{algorithmic}[1]
		\Require A tree $T$ rooted at $s$, two values $N, K$, two weight vectors $w$ and $u$, a parameter $\lambda$.
		\Ensure The optimal value $h(s, 1:|S(s)|, N, \lambda)$, an optimal solution $\hat w_{(s, 1:|S(s)|, N, \lambda)}$.
              \State Call \textbf{DFS}$(s, 1:|S(s)|, N, \lambda)$,
            \State \Return $h(s, 1:|S(s)|, N, \lambda), \hat w_{(s, 1:|S(s)|, N, \lambda)}$.
            \Function{\textbf{DFS}}{$v,p:q,N,\lambda$}
            \If{$v \in L(s)$ or $q<p$ or $N< 0$}
            \State \Return $SP_{(v, p: p, N, \lambda)}:=0$,  $SRD_{(v, p:p, N, \lambda)}:=0$, $h(v, p: p, N, \lambda):=0$,  $E_{(v,p:p, N, \lambda)}:=\emptyset$, $\hat w_{(v, p: p, N, \lambda)}:= w$ .
            \EndIf
            \If{$p=q$}
            \State Calculate \textbf{DFS}($v,p:p, N, \lambda)$, \textbf{DFS}($v_{s_p},p:p, N-r(e_{s_p}), \lambda)$.
            \State By (\ref{DITH_2eq_b}), let $h_1:=h\left(v_{s_{p}}, 1:|S(v_{s_{p}})|, N-N_1, \lambda\right)+\lambda w(e_{s_p})+(1-\lambda) |L(e_{s_p})| w(e_{s_p})$, $h_2:=h\left(v_{s_{p}}, 1:|S(v_{s_{p}})|, N-N_1-r(e_{s_p}), \lambda\right)+
\lambda \bar{w}(e_{s_p})+(1-\lambda) |L(e_{s_p})| \bar{w}(e_{s_p})$.
            \If{$h_1<h_2$}
            \State \Return $SP_{(v, p: p, N, \lambda)}:=SP_{(v_{s_{p}}, 1:|S(v_{s_{p}})|, N-r(e_{s_p}), \lambda)}+\bar w(e_{s_p})$,  $SRD_{(v, p:p, N, \lambda)}$ $:=SRD_{(v_{s_{p}}, 1:|S(v_{s_{p}})|, N-r(e_{s_p}), \lambda)} +|L(e_{s_p})|\bar w(e_{s_p})$, $h(v, p: p, N, \lambda):=h_2$,  $E_{(v,p:p, N, \lambda)}:= E_{(v_{s_{p}}, 1:|S(v_{s_{p}})|,N-r(e_{s_p}), \lambda)} \cup \{ e_{s_p}\}$, $\hat w_{(v,p:p, N, \lambda)}(e):= \begin{cases}
                \bar w(e), e \in E_{(v,p:p, N, \lambda)}, \\
                 w(e), otherwise.            \end{cases}$
            \Else
            \State \Return $SP_{(v, p: p, N, \lambda)}:=SP_{(v_{s_{p}}, 1:|S(v_{s_{p}})|, N, \lambda)}+ w(e_{s_p})$,  $SRD_{(v, p:p, N, \lambda)}:=SRD_{(v_{s_{p}}, 1:|S(v_{s_{p}})|, N, \lambda)} +|L(e_{s_p})|w(e_{s_p})$, $h(v, p: p, N, \lambda):=h_1$,  $E_{(v,p:p, N, \lambda)}:= E_{(v_{s_{p}}, 1:|S(v_{s_{p}})|,N, \lambda)}$, $\hat w_{(v,p:p, N, \lambda)}(e):= \begin{cases}
                \bar w(e), e \in E_{(v,p:p, N, \lambda)}, \\
                 w(e), otherwise.
            \end{cases}$
            \EndIf
            \Else
            \For{$N_1 = 0$ to $N$}
            \State Calculate \textbf{DFS}($v,1:p-1 , N_1, \lambda)$, \textbf{DFS}($v,p:p , N-N_1, \lambda)$.
            \EndFor
            \State Calculate $h(v,1:p,N, \lambda)$ by (\ref{DITH_2eq_a}) whose maximum achieved at  $N_1:=N_1^*$.
            \State \Return $SP_{(v,1:p,N,\lambda)}:= \min \{SP_{(v,1:p-1,N_1^*,\lambda)},SP_{(v,p:p,N-N_1^*,\lambda)} \}$, $SRD_{(v,1:p,N, \lambda)}:=SRD_{(v,1:p-1,N_1^*,\lambda)}+SRD_{(v,p:p,N-N_1^*,\lambda)}$, $h(v,1:p,N, \lambda)$, $E_{(v,1:p,N, \lambda)}:=E_{(v,1:p-1,N_1^*,\lambda)} \cup E_{(v,p:p,N-N_1^*, \lambda)}$, $\hat w_{(v,1:p, N, \lambda)}(e):= \begin{cases}
                \bar w(e), e \in E_{(v,1:p, N, \lambda)}, \\
                 w(e), otherwise.
                 \end{cases}$
            \EndIf
            \EndFunction

	\end{algorithmic}
\end{algorithm}
\begin{theorem} \label{thm-CIT}
    Algorithm  \ref{Alg_PIC_inf} can solve problem ($CIT_{H \infty}^{\lambda}$) in $O(nN^2)$ time.
\end{theorem}
\begin{proof}
     Given that $N$ is a predefined constant, for each node $v$ and each interger $0 \le N_1 \le N$, there is a distinct state \textbf{DFS$(v,1:|S(v)|,N_1,\lambda)$}. Consequently, the total number of states is bounded by $O(nN)$. For each state, \textbf{DFS} function delineated in Algorithm \ref{Alg_PIC_inf} requires $O(N)$ time. Therefore, the overall computational complexity for problem ($CIT_{H \infty}^{\lambda}$) is $O(nN^2)$.
\end{proof}

\subsubsection{A binary search algorithm to solve problem (DIT$_{H\infty}$)}

To solve problem (DIT$_{H\infty}$), we first analyse the connection between problem ($CIT_{H \infty}^{\lambda}$) and (DIT$_{H\infty}$), then propose the montonicity theorem of problem ($CIT_{H \infty}^{\lambda}$). Based on which, we develop a binary search algorithm.
 
Note that by setting $\lambda=1$ or $0$, we are able to derive optimal solutions for two planning problems: maximizing the StRD and maximizing SRD, subject to upper bound and weighted sum Hamming distance constraints, respectively. Then the following theorem concerning infeasibility can be obtained.

\begin{theorem} \label{Thm-DIT_Hinf}
If $h(s, 1:|S(s)|, N, 1) < M$, then problem (DIT$_{H\infty}$) is infeasible.
\end{theorem}
\begin{proof}
Note that $h(s, 1:|S(s)|, N, 1)=SP_{(s, 1:|S(s)|, N, 1)} < M$, then there is no set of edges capable of extending the StRD to meet or exceed $M$, which means the problem is infeasible.
\end{proof}

Analysing problem (DIT$_{H\infty}$), we need to strike a balance between these objectives StRD and SRD. Specifically, within the original constraints, we seek to establish a lower bound for the StRD while simultaneously maximizing SRD. Given that the objective function $h$, comprises two components, we observe that as the parameter $\lambda$ varies from 1 to 0, the contribution of the StRD to $h$ steadily diminishes, while the contribution of the SRD progressively increases. Despite the inherent discontinuity imposed by the weighted sum Hamming distance constraints, as $\lambda$ changes, the algorithm becomes biased in the selection of update edges, which allows us to draw the following theorems.

\begin{theorem} Let $1\ge \lambda_1 \ge \lambda_2\ge 0$, then $SP_{(s,1:|S(s)|,N,\lambda_1)} \ge SP_{(s,1:|S(s)|,N,\lambda_2)}$, $SRD_{{(s,1:|S(s)|,N,\lambda_1)}} \le SRD_{{(s,1:|S(s)|,N,\lambda_2)}}$.
\end{theorem}
\begin{proof}
Let $w_1:=\hat w_{(s,1:|S(s)|,N,\lambda_1)}$ and $w_2:=\hat w_{(s,1:|S(s)|,N,\lambda_2)}$ be optimal solutions with respect to $\lambda_1$ and $\lambda_2$, respectively. Let $P_1$ and $P_2$  be the corresponding shortest path under $w_1$ and  $w_2$,  respectively. Let $SP_{w_i}=SP_{(s,1:|S(s)|,N,\lambda_i)}$ and $SRD_{w_i}=SRD_{(s,1:|S(s)|,N,\lambda_i)},i=1,2$ for simplicity.

%When $\lambda=\lambda_1$, let $w_1:=\hat w_{(v,1:p,N,\lambda_1)}$ an optimal solution and $P_1$ the shortest path under $w_1$, similarly define $w_2:=\hat w_{(v,1:p,N,\lambda_2)}$ and $P_2$ when $\lambda=\lambda_2$.

By Theorem \ref{thm-DIT-NPH}, problem ($CIT_{H \infty}^{\lambda}$) can be transformed to a 0-1 knapsack problem ($TCIT_{H \infty}^{\lambda}$), for convenience, we divide $1-\lambda$ in both terms of the objective function. Let $x(e)=1,0$ represent whether the weight of  edge $e$ is upgraded to $\bar w(e)$ or not. Then we have
\begin{eqnarray}
&\max \limits & \left\{  \sum_{e \in E} |L(e)|(x(e)(\bar w(e)-w(e))+w(e))  +\nn \right. \\ 
&&\left. \frac{\lambda}{1-\lambda} \min \limits_{t_k \in Y } \sum_{e \in P_k} (x(e)(\bar w(e)-w(e))+w(e)) \right\} \nn\\
(\textbf{TCIT}_{H\infty}^{\lambda}) 
&s.t.&   \sum_{e\in E}r(e)x(e)\le N, \nn \\
&& x(e)=1,0.
\end{eqnarray}
There are two cases to be considered.

\textbf{Case 1: $P_1 = P_2$.} In this case, the increment of the objective function resulting from upgrading edge $e$ on path $P_1$ is monotonically increasing as $\lambda$ varies from 0 to 1 since $\frac{\lambda}{1-\lambda}$ is monotonically increasing, while the increment of the objective function by upgrading other edges not on $P_1$ keeps still. When $w_1 = w_2$, the theorem holds trivially. Suppose $w_1 \neq w_2$. Since $w_2$ is the optimal solution when $\lambda=\lambda_2$, in order for $w_1$ to surpass $w_2$, it must select edges that can increase higher values. Consequently, more edges on $P_1$ will be chosen under $w_1$, leading to $SP_{w_1} > SP_{w_2}$.

\textbf{Case 2: $P_1 \neq P_2$.} Following the insights from Case 1, where identical shortest paths result in a non-decreasing StRD as $\lambda$ increases, the scenario where $P_1 \neq P_2$ implies that enhancements to $P_1$ have been so substantial that it is no longer the shortest path, whereas $P_2$ becomes so. Under these circumstances, $SP_{w_1} > SP_{w_2}$ is ensured. 

In both cases, we get $SP_{w_1} \ge SP_{w_2}$, now we show $SRD_{w_2} \ge SRD_{w_1}$. By the optimality of $w_2$, we have 
\begin{eqnarray}
    \lambda_2 SRD_{w_1}+(1-\lambda_2)SP_{w_1}\le \lambda_2 SRD_{w_2}+(1-\lambda_2)SP_{w_2}\label{eq_opt_1}
   % \\\lambda_1 SRD_{w_2}+(1-\lambda_1)SP_{w_2}\le \lambda_1 SRD_{w_1}+(1-\lambda_1)SP_{w_1} \label{eq_opt_2}
\end{eqnarray}
Rearrange the inequality \eqref{eq_opt_1}, then we have
\begin{equation}\label{eq_opt_2}
    \lambda_2 (\text{SRD}_{w_1} - \text{SRD}_{w_2}) \leq (1-\lambda_2)(\text{SP}_{w_2} - \text{SP}_{w_1})
\end{equation}

Since \( (1-\lambda_2) \geq 0 \) and \( \text{SP}_{w_2} \leq \text{SP}_{w_1} \), the right-hand side of \eqref{eq_opt_2} is non-positive, thus:
\begin{equation}
    \lambda_2 (\text{SRD}_{w_1} - \text{SRD}_{w_2}) \leq 0, \nn
\end{equation}
which leads to
\[
\text{SRD}_{w_1} \leq \text{SRD}_{w_2}.
\]
This completes the proof.
\end{proof}

\begin{theorem} \label{thm-lambda-13}
 If $h(s, 1:|S(s)|, N, 1) \ge M$, then the optimal solution of the problem (DIT$_{H\infty}$) is given by $\hat w_{(s, 1:|S(s)|, N, \lambda^*)}$, where $\lambda^*$ is the critical value with $0\leq \lambda^*\leq 1$.
\end{theorem}

\begin{proof}
When \(h(s, 1:|S(s)|, N, 1) \ge M\), then the problem is feasible by Theorem \ref{Thm-DIT_Hinf}. 

Notice that the StRD is non-increasing and the SRD is non-decreasing as $\lambda$ changes from $1$ to $0$.  Besides, the StRD and SRD reaches the maximal at $\lambda=1$ and $0$, respectively. Hence,  a solution $\hat w_{(s, 1:|S(s)|, N, \lambda^*)}$ emerges at a critical threshold denoted as $\lambda^*$, which not only adheres to the StRD constraint but also optimally maximizes SRD. Then $w^*:=\hat w_{(s, 1:|S(s)|, N, \lambda^*)}$ is the optimal solution for problem (DIT$_{H\infty}$).

Suppose there exists a better solution $\tilde w$ than $w^*$ with $\tilde w(T) > w^*(T)$. Observed that $\tilde w$ also satisfies the feasibility of problem ($CIT_{H \infty}^{\lambda^*}$) and $w^*$ have the largest $h$ value under feasibility by Theorem \ref{thm-CIT}, then 
$$\lambda^* \min _{t_k \in Y} \tilde w(P_k) +(1-\lambda^*)\tilde w(T) \le \lambda^* \min _{t_k \in Y} w^*(P_k) +(1-\lambda^*) w^*(T), \nn$$ which leads to $\min _{t_k \in Y} \tilde w(P_k)< \min _{t_k \in Y} w^*(P_k)$ since $\tilde w(T) > w^*(T)$.
Note that $\lambda$ is continuous and $\lambda^*$ is the threshold,
 there is no room for the StRD of $w^*$ to decrease, which is a contradiction. 
\end{proof}
 %By feasibility of $\hat w$ on problem (DIT$_{H\infty}$), we have $\sum_{e \in E} r(e)H(\hat w(e), w(e)) \le N$, $w(e)\le \hat w(e) \le \bar w(e), e \in E$, $\min_{t_k \in Y} \hat{w}(P_k)\ge M$. 
 
% This contradiction in build on the assumption that $\hat w$ is mapped to a optimal solution of problem ($CIT_{H \infty}^{\lambda}$) with $\lambda=\hat \lambda$ which is true since there is always some $\hat \lambda$ that can achieve $SP_{(s,1:|S(s)|,N,\tilde \lambda)}= \min _{t_k \in Y} \hat w(P_k)$, and because $SRD_{\hat w}$ is the largest in this case, then $h(s,1:|S(s)|,N,\tilde \lambda)$ must equal to $ \lambda  \min_{t_k \in Y} \hat w(P_k) +(1-\hat \lambda) \hat w(T)$.

Finally, we present Algorithm \ref{Alg_DITH_inf} to solve (DIT$_{H\infty}$), encapsulating our findings. This algorithm leverages a binary search method to determine the critical value $\lambda^*$ corresponding to Theorem \ref{thm-lambda-13}. Then we call Algorithm \ref{Alg_PIC_inf} to solve problem (CIT$_{H \infty}^{\lambda^*}$), in which   \textbf{DFS}$(s, 1:|S(s)|, N, \lambda^*)$ is utilized.

\begin{algorithm}
	\caption{A binary search algorithm to solve (DIT$_{H\infty}$).}
 \label{Alg_DITH_inf}
	\small
	\begin{algorithmic}[1]
		\Require A tree $T$ rooted at $s$, two given values $N, K$, two edge weight vectors $w$ and $u$.
		\Ensure An optimal solution $\tilde{w}$, the optimal value $\tilde{w}(T)$.
            \State Let $U:=\min_{t_k \in Y} u(P_k) + \sum_{e \in E} |L(e)| u(e)$.
            \State Calculate $\bar w$ by Equation (\ref{DIT_inf_ans}), let $p^1:=SP_{(s, 1:|S(s)|, N, 1)}$ by calling \textbf{DFS}$(s, 1:|S(s)|, N, 1)$  in Algorithm \ref{Alg_PIC_inf}.
            \If{$ p^1< M$}
            \State \Return No feasible solution.
            \Else
            \State Let $lr:=0, rr:=1$.
            \While{$rr-lr> \frac{1}{U^2}$}
            \State Let $mid:= \frac{(lr+rr)}{2}$, $p^{mid}:=SP_{(s, 1:|S(s)|, N, mid)}$  by calling \textbf{DFS}$(s, 1:|S(s)|, N, mid)$ in Algorithm \ref{Alg_PIC_inf}.
            \If{$p^{mid}>M$ }
                \State \(rr:=mid\)
                 \Else
                    \State \(lr:= mid\)
                \EndIf
            \EndWhile
            \State Let $\lambda^*:=rr$. Call \textbf{DFS}$(s, 1:|S(s)|, N, \lambda^*)$.
            \State \Return $\tilde{w}:=\hat w_{(s, 1:|S(s)|, N, \lambda^*)}, \tilde{w}(T):=SRD_{(s, 1:|S(s)|, N, \lambda^*)}$.
            \EndIf
	\end{algorithmic}
\end{algorithm}

\begin{theorem}\label{thm:14}
    Let U:=$\min_{t_k \in Y}  u(P_k)+\sum_{e\in E}|L(e)| u(e)$. When $N \ge 1$,  Algorithm \ref{Alg_DITH_inf} can solve problem (DIT$_{H\infty}$) within $O(nN^2\log_2 U)$ time.
\end{theorem}
\begin{proof}
    By Theorem \ref{thm-lambda-13}, there exists a critical value $\lambda^*\in (lr^*,rr^*)$ that gives the optimal solution. To be specific, there exists a small interval $[a,b] $ satisfying $(lr^*,rr^*)\subset [a,b] \subset [0,1]$ such that $\hat{w}_{(s,1:|S(s)|,N,\lambda^*)}=\hat{w}_{(s,1:|S(s)|,N,\lambda)}$ holds for any $\lambda \in [a,b]$. Therefore, suppose $b-a=\epsilon$, for the binary search process, it takes at most $O(|\log_2 \epsilon|)$ iterations and it runs $O(nN^2)$ in each iteration by Theorem \ref{thm-CIT}. Hence the time complexity is $O(nN^2|\log_2 \epsilon|)$.
    
    Here we give a lower bound for $\epsilon$.
    Without loss of generosity, let us assume $w$ and $\bar w$ are integer vectors since they are represented as decimal in calculation. Then, for problem ($CIT_{H \infty}^{\lambda}$) in \eqref{PIC_inf}, $h$ is a piecewise linear function of $\lambda$ coloured in red as shown in Fig.  \ref{fig:prove}.
    \begin{figure}[htbp]
    \centering
    \includegraphics[width=0.3\textwidth]{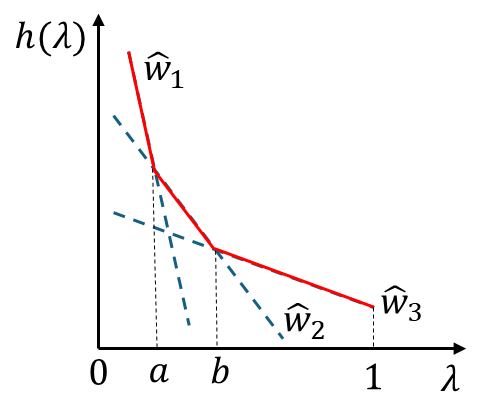}
    \caption{The relationship of $\lambda$ and $h(\lambda)$.}
    \label{fig:prove}
    \end{figure}

Each feasible edge upgrade vector $\hat{w}_i$ can be represented as a non-increasing line segment by
\begin{equation}
    h( \lambda) = \left(\min_{t_k \in Y} \hat{w}_i(P_k) - \sum_{e \in E} |L(e)| \hat{w}_i(e)\right) \lambda + \sum_{e \in E} |L(e)| \hat{w}_i(e).
\end{equation}

By analyzing, we can obtain
\begin{eqnarray*}
  &&  -\sum_{e \in E} |L(e)| \bar{w}(e) < a_i:=\min_{t_k \in Y} \hat{w}_i(P_k) - \sum_{e \in E} |L(e)| \hat{w}_i(e) < \min_{t_k \in Y} \bar{w}(P_k), \\
 &&   0 \leq b_i:=\sum_{e \in E} |L(e)| \hat{w}_i(e) \leq \sum_{e \in E} |L(e)| \bar{w}(e).
\end{eqnarray*}

For each $\lambda \in [0, 1]$, it corresponds to a problem ($CIT_{H \infty}^{\lambda}$) where the optimal solution output by Algorithm \ref{Alg_PIC_inf} is the highest among all feasible edge weight vectors. For any two lines, we have
\begin{eqnarray*}
    h_i(\lambda) = a_i \lambda + b_i, \\
    h_j(\lambda) = a_j \lambda + b_j.
\end{eqnarray*}
The horizontal coordinate of their intersection point is
\[ \lambda_1 = \frac{b_j - b_i}{a_i - a_j}. \]

Similarly, let \(\lambda_2 = \frac{b_j' - b_i'}{a_i' - a_j'}\) represents the horizontal coordinate of another intersection point. Consequently, we get a lower bound of $b-a$ as
\[
b - a \geq \frac{b_j - b_i}{a_i - a_j} - \frac{b_j' - b_i'}{a_i' - a_j'} \geq \frac{1}{(a_i - a_j)(a_i' - a_j')} > \frac{1}{(\min_{t_k \in Y} \bar{w}(P_k) + \sum_{e \in E} |L(e)| \bar{w}(e))^2}.
\]

We also know:
\[
\min_{t_k \in Y} \bar{w}(P_k) + \sum_{e \in E} |L(e)| \bar{w}(e) \leq U := \min_{t_k \in Y} u(P_k) + \sum_{e \in E} |L(e)| u(e).
\]

Therefore, we establish
\[
|\log_2(b - a)| > |\log_2 \frac{1}{U^2}|,
\]
from which it follows that the algorithm terminates within $O(nN^2 |\log_2 \frac{1}{U^2}|)=O(2nN^2 \log_2 U)=O(nN^2 \log_2 U)$. This runtime is classified as pseudo-polynomial because \(N\) is dependent on the input provided. 
Specifically, when \( r(e) = 1, N < n \), the weighted sum of Hamming distances becomes a cardinality constraint. This algorithm runs in \( O(n^3 \log_2 U) \) time, making it a polynomial-time algorithm.
\end{proof}

\section{The $\mathcal{NP}$-Hardness of (MCDITH) under the Weighted $l_\infty$ Norm}

Next, we consider its related minimum cost problem (MCDIT$_{H\infty}$) by exchanging the objective function and the $l_\infty$ norm of problem (DIT$_{H\infty}$), which is formulated in (\ref{MCDITH_inf}).
To prove (MCDIT$_{H\infty}$) is $\mathcal{NP}$-hard, it suffices to show that its decision version (DMCDIT$_{H\infty}$) is $\mathcal{NP}$-complete. Typically, proving a decision problem is $\mathcal{NP}$-complete involves two steps: first, demonstrating that the decision version is in $\mathcal{NP}$; second, showing that the decision version can be reduced from a problem already proven to be $\mathcal{NP}$-complete. For our purposes, we choose the decision version of the 0-1 knapsack problem as the original problem \cite{Combia book}. This leads to the following theorem.

\begin{theorem}
    The problem (MCDIT$_{H\infty}$)  is $\mathcal{NP}$-hard.
\end{theorem}

\begin{proof}
We prove the theorem in two steps.

\textbf{Step 1:} The decision version of the (MCDIT$_{H\infty}$) problem is in $\mathcal{NP}$.

The decision version of (MCDIT$_{H\infty}$) can be stated as follows: Given a maximal cost $K$, determine whether there exists an edge weight vector $\hat{w}$ that satisfies the following constraints:

\begin{eqnarray}
    &&\min 1 \nonumber\\
    &\text{s.t.}& \hat{w}(T) \geq D, \nonumber\\
   (DMCDITH_\infty) && \min \limits_{t_k \in Y} \hat{w}(P_k) \ge M, \label{DMCDITH} \\
    && \sum_{e\in E} r(e)H(\hat{w}(e), w(e)) \le N, \nonumber\\
    && \max_{e\in E}c(e) (\hat w(e) - w(e)) \le K, \nonumber\\
    && w(e) \leq \hat{w}(e) \leq u(e), \quad e \in E. \nonumber
\end{eqnarray}

Note that under the maximal cost $K$, the weight of one edge is constrained to the upper bound in Equation (\ref{DIT_inf_ans}), thus the $u(e)$ in the constraint can be replaced by $\bar{w}(e)$. By the definition of an $\mathcal{NP}$ problem, given the vector $\bar{w}$, one can easily verify whether the vector satisfies the constraints in $O(n)$ time, as it only requires basic vector operations.

\textbf{Step 2:} Problem (DMCDIT$_{H\infty}$) can be reduced from the 0-1 knapsack problem.

Similar to (DIT$_{H\infty}$), observe that when the upgrade cost is fixed, it is always better to upgrade an edge to its upper bound $\bar{w}(e)$ in Equation (\ref{DIT_inf_ans}). Therefore, the decision version of (MCDIT$_{H\infty}$) is equivalent to the following problem, where $In_{\bar{w}}(e)$ represents the SRD increment when upgrading edge $e$ to $\bar{w}(e)$, $x(e)$ represents whether to upgrade edge $e$.
\begin{eqnarray*}
    & \max & 1 \\    \text{\textbf{(GMCDIT$_{H\infty}$)}} \; & \text{s.t.} & \sum_{e \in E} r(e)x(e) \leq N, \\
    && \sum_{e \in E} In_{\bar{w}}(e)x(e) \ge D, \\
    && \sum_{e \in P_k } (\bar{w}(e)-w(e)) x(e) \geq M -w(P_k), \forall k=1,\dots , m,\\
    && x(e) \in \{ 0,1\}. 
\end{eqnarray*}

Conversely, the decision version of the 0-1 knapsack problem can be formulated as follows: Given a constant $M_1$, find a feasible solution of 
\begin{equation}
\begin{aligned}
\max \quad & 1 \\
\text{s.t.} \quad & \sum_{i=1}^{n} c_{i} x_{i} \le R, \\
& \sum_{i=1}^{n} p_{i} x_{i} \ge M_1,\\ 
&x_{i}\in \{ 0,1 \}, \quad\forall i=1,\dots,n.
\end{aligned}
\label{decision:01knapsack}
\end{equation}
Then for any instance $I_1$ of form (\ref{decision:01knapsack}), consider a chain with root $s$, then $|L(e)|=1, e \in E$, by setting $r(e_i) =c_i, In_{\bar w}(e)=p_i, D=M_1, N=R$, instance $I_2$ of problem \textbf{(DMCDIT$_{H\infty}$)} is exacly instance $I_1$. Therefore the decision version of (MCDIT$_{H\infty}$) can be reduced from the 0-1 knapsack problem.

In conclusion, the (MCDIT$_{H\infty}$) problem is $\mathcal{NP}$-hard.
\end{proof}

\section{An pseudo-polynomial  time algorithm to solve problem (MCDIT$_{H\infty}$)}

For problem (MCDIT$_{H\infty}$), the objective function, which represents the maximum cost under the $l_\infty$ norm, is actually constrained within the interval $[K^1,K^2]$, where $K^1=0$, and $K^2= \max_{e \in E} c(e)(u(e)-w(e))$. Upon investigating the interconnection between problems (DIT$_{H\infty}$) and (MCDIT$_{H\infty}$), it is discovered that the latter can be addressed by sequentially solving (DIT$_{H\infty}$) to ascertain the minimum $K^*$ such that $\tilde w(T) > D$. Moreover, in problem (DIT$_{H\infty}$), it is observed that as the value of $K$ ascends, the optimal SRD value, $\tilde w(T)$ also exhibits a monotonic increase.

To find optimal $K^*$ within the defined interval $[K^1, K^2]$, a binary search Algorithm is employed. Each phase of the iteration leverages Algorithm \ref{Alg_DITH_inf} for the computation of $\tilde w(T)$, thus iteratively pinpointing $K^*$. Consequently, we have the following theorem.
\begin{theorem}
    Problem (MCDIT$_{H\infty}$) can be solved by Algorithm \ref{Alg_MCDITH_inf} within $O(nN^2 \log_2 K\log_2 U)$ time.
\end{theorem}
\begin{proof}
By Theorem \ref{thm:14}, Algorithm \ref{Alg_DITH_inf} operates in \(O(nN^2 \log_2 U)\) time. Suppose $K$ is integer, the binary search method requires at most \(O(\log_2 K)\) iterations. Consequently, Algorithm \ref{Alg_MCDITH_inf} has a time complexity of \(O(nN^2 \log_2 K \log_2 U)\), which qualifies it as a pseudo-polynomial time algorithm depending on \(N\).
\end{proof}

\begin{algorithm}
	\caption{A pseudo-polynomial  time algorithm to solve problem (MCDIT$_{H\infty}$)}
 \label{Alg_MCDITH_inf}
	\small
	\begin{algorithmic}[1]
		\Require A tree $T$ rooted at $s$, a given value $ N$, two edge weight vectors $w$ and $u$.
		\Ensure Minimal cost value $K^*$, a set of upgrade edges $E^*$ and the optimal solution $ \tilde{w}$.
            \State  Let $K^1:=0$ and $K^2:=\max_{e \in E} c(e)(u(e)-w(e))$. 
            \State Let $y$ the optimal SRD value of DIT$_{H\infty}$($T,N,K^2, w,u$)
            \If{$ y< D$}
            \State \Return No feasible solution.
            \Else
            \While{$K^2-K^1> 1$}
            \State Let $mid:= \frac{(K^2+K^1)}{2}$, $SRD^{mid}$ the SRD returned by DIT$_{H\infty}$($T,N,mid, w,u$).
            \If{$SRD^{mid}>D$ }
                \State \(K^2:=mid\)
                 \Else
                    \State \(K^1:= mid\)
                \EndIf
            \EndWhile
            \State \Return $K^2$, the upgrade edge set $E^*$ returned by DIT$_{H\infty}$($T,N,K^2, w,u$), the optimal solution $\tilde w$ returned by DIT$_{H\infty}$($T,N,K^2, w,u$).
            \EndIf
	\end{algorithmic}
\end{algorithm}

 \section{Numerical Experiments}
\subsection{One example to show the process of Algorithm \ref{Alg_DITH_inf}}
For the better understanding of Algorithm \ref{Alg_DITH_inf}, Example \ref{example_1} is given to show the detailed computing process.

\begin{example} \label{example_1}

    Let $V:=\left \{ s,v_1,\dots,v_{12} \right \}$, $E:=\left \{  e_1,\dots,e_{12}\right \}$, the corresponding $i,c,w,\bar w$ are shown on edges with different colors. Now we have $w(T):=106$, $u(T):=216$. Suppose the given values are $M:=30,K=320$ and $N:=3$, $r(e)=1, \forall e \in E$.  
    \begin{figure}[htbp!!]
    \centering
    \includegraphics[width=1.0\linewidth]{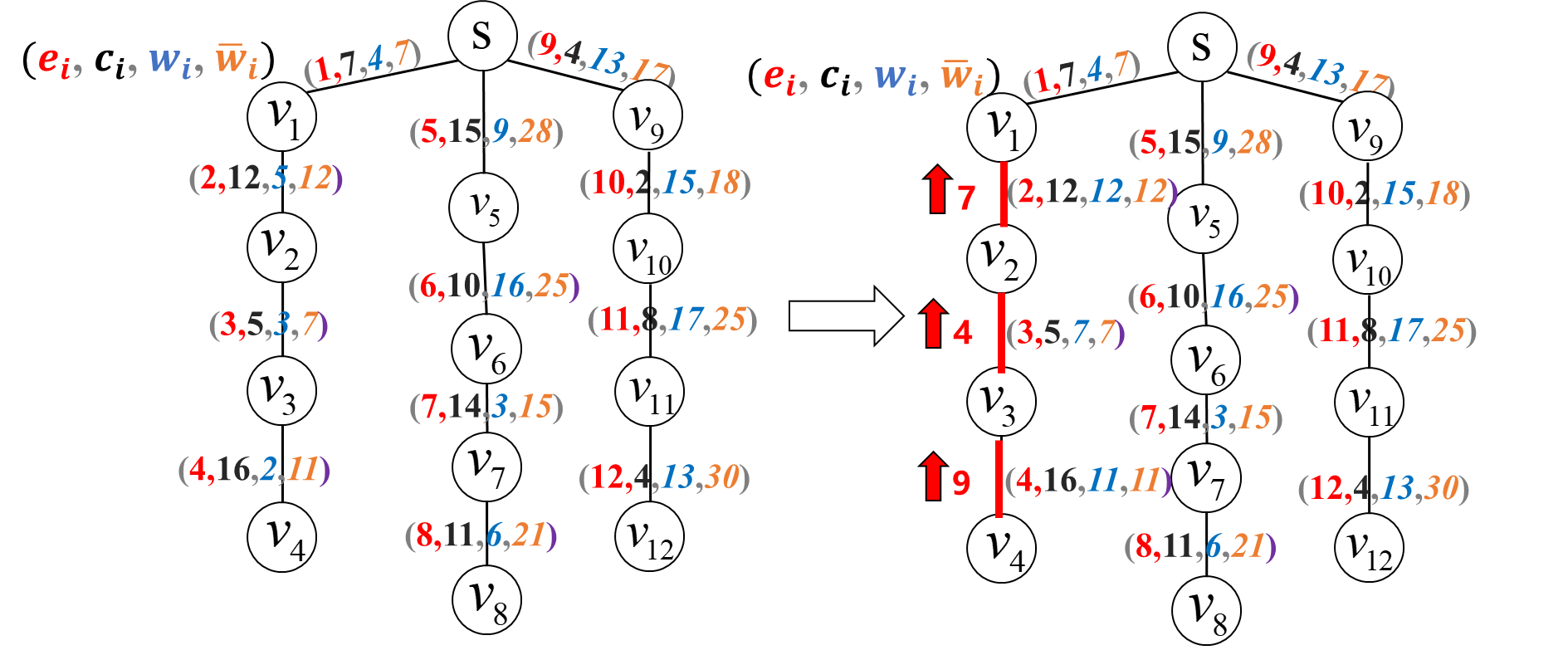}
    \caption{The left figure shows the weights $(e_i,c_i,w_i, \bar w_i)$. The right figure shows the upgraded edges in Alg. \ref{Alg_DITH_inf} when $\lambda=1$.}
    \label{example_1_fig}
\end{figure}
When $\lambda=1$, the algorithm aims to max StRD as shown in Fig.\ref{example_1_fig} where all upgrade edges are on path $(s,v_4)$. Table \ref{tab:1} shows that only on the shortest path h have positive values.
\begin{table}[htbp!!]
\centering
\label{tab:1}
\caption{Values of $h$ value for \( \lambda = 1 \)}
\begin{tabular}{|c|c|c|c|}
\hline
\( h \)           & Value & $h$ & Value \\ \hline
\( h(v_3, 1:1,1) \) & 11  & \( h(s, 1:1,1) \)   & 23 \\
\( h(v_2, 1:1,1) \) & 14  & \( h(s, 1:1,2) \)   & 30  \\
\( h(v_2, 1:1,2) \) & 18  & \( h(s, 1:1,3) \)   & 34 \\
\( h(v_1,1:1,1) \) & 19  & \( h(s, 1:2,3) \)   & 34 \\
\( h(v_1, 1:1,2) \) & 26  & \( h(s, 3:3,3) \)   & 0  \\
\( h(v_1, 1:1,3) \) & 30  & \( h(s, 1:3,3) \)   & 34  \\

\hline
\end{tabular}
\end{table}
In $\hat w_{(s,1:3,3,1)}$, SP$_{(s,1:3,3,1)}=34 \ge M=30$, so the problem is feasible.
    
    \begin{figure}[htbp!!]
    \centering
    \includegraphics[width=1.0\linewidth]{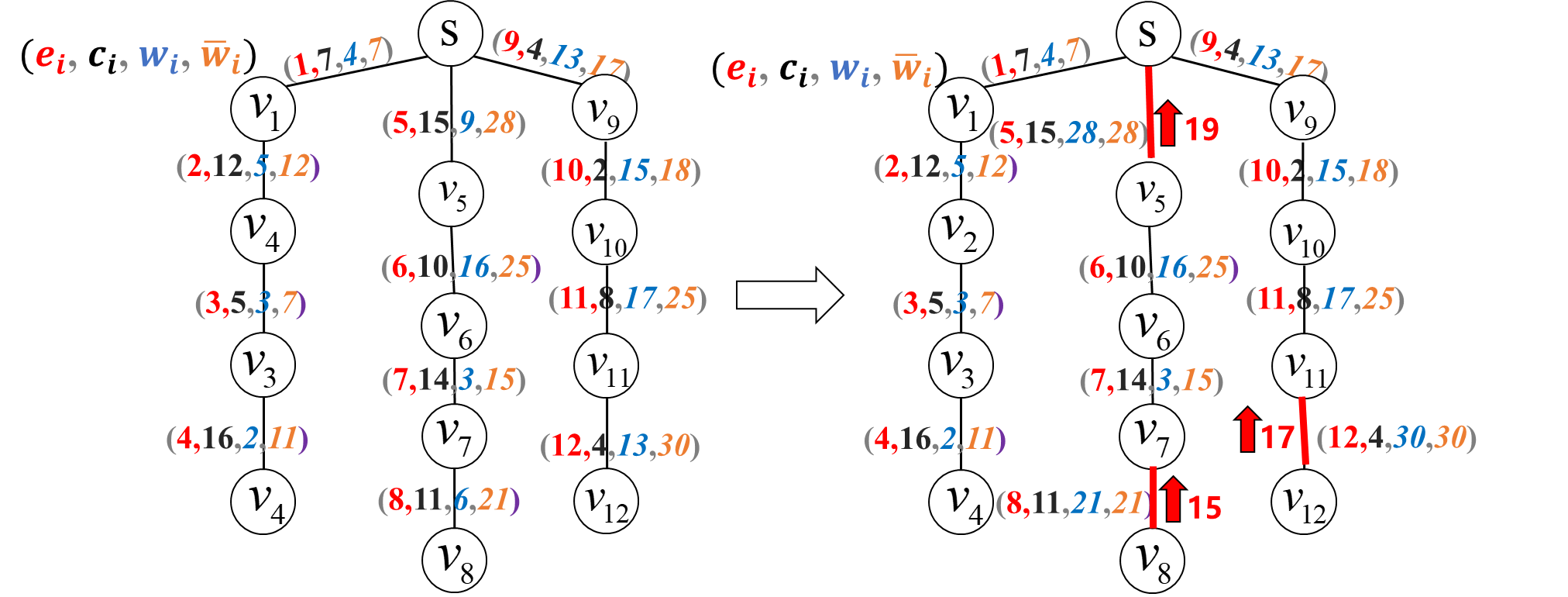}
    \caption{The left figure shows the weights $(e_i,c_i,w_i, \bar w_i)$. The right figure shows the upgraded edges in Alg. \ref{Alg_DITH_inf} when $\lambda=0$.}
    \label{example_2_fig}
\end{figure}

Conversely, when $\lambda=0$, the algorithm maximizes SRD, as shown in Fig.\ref{example_2_fig}. In $\hat w_{(s,1:3,3,0)}$, the edges with larger SRD increment are chosen. Table \ref{tab:2}  shows that the algorithm maximizes h in all feasible solution.
\begin{table}[htbp!!]
\centering
\caption{Values of $h$ value for \( \lambda = 0 \)}
\begin{tabular}{|c|c|c|c|}
\hline
\( h \)           & Value & $h$ & Value \\ \hline
\( h(v_3, 1:1,1) \) & 11  & \( h(s, 1:1,1) \)   & 23 \\
\( h(v_2, 1:1,1) \) & 14  & \( h(s, 1:1,2) \)   & 30  \\
\( h(v_2, 1:1,2) \) & 18  & \( h(s, 1:1,3) \)   & 34 \\
\( h(v_1,1:1,1) \) & 19  & \( h(s, 1:2,3) \)   & 94 \\
\( h(v_1, 1:1,2) \) & 26  & \( h(s, 3:3,3) \)   & 87  \\
\( h(v_1, 1:1,3) \) & 30  & \( h(s, 1:3,3) \)   & 157  \\
\hline
\end{tabular}
\label{tab:2}
\end{table}

By adjusting $\lambda$, different weight vector are generated. In the end, when $\lambda=0.75$, the algorithm terminates with the optimal value $\tilde w(T)=141$ in Fig.\ref{example_3_fig} and Table \ref{tab:3}.

\begin{table}[htbp!!]
\centering
\caption{Values of $h$ value for \( \lambda = 0.75 \)}
\begin{tabular}{|c|c|c|c|}
\hline
\( h \)           & Value & $h$ & Value \\ \hline
\( h(v_3, 1:1,1) \) & 11  & \( h(s, 1:1,1) \)   & 23 \\
\( h(v_2, 1:1,1) \) & 14  & \( h(s, 1:1,2) \)   & 30  \\
\( h(v_2, 1:1,2) \) & 18  & \( h(s, 1:1,3) \)   & 34 \\
\( h(v_1,1:1,1) \) & 19  & \( h(s, 1:2,3) \)   & 43.25 \\
\( h(v_1, 1:1,2) \) & 26  & \( h(s, 3:3,3) \)   & 21.75  \\
\( h(v_1, 1:1,3) \) & 30  & \textcolor{red}{\( h(s, 1:3,3) \)}  & \textcolor{red}{57.75}  \\
\hline
\end{tabular}
\label{tab:3}
\end{table}

\begin{figure}[htbp!!]
    \centering
    \includegraphics[width=1.0\linewidth]{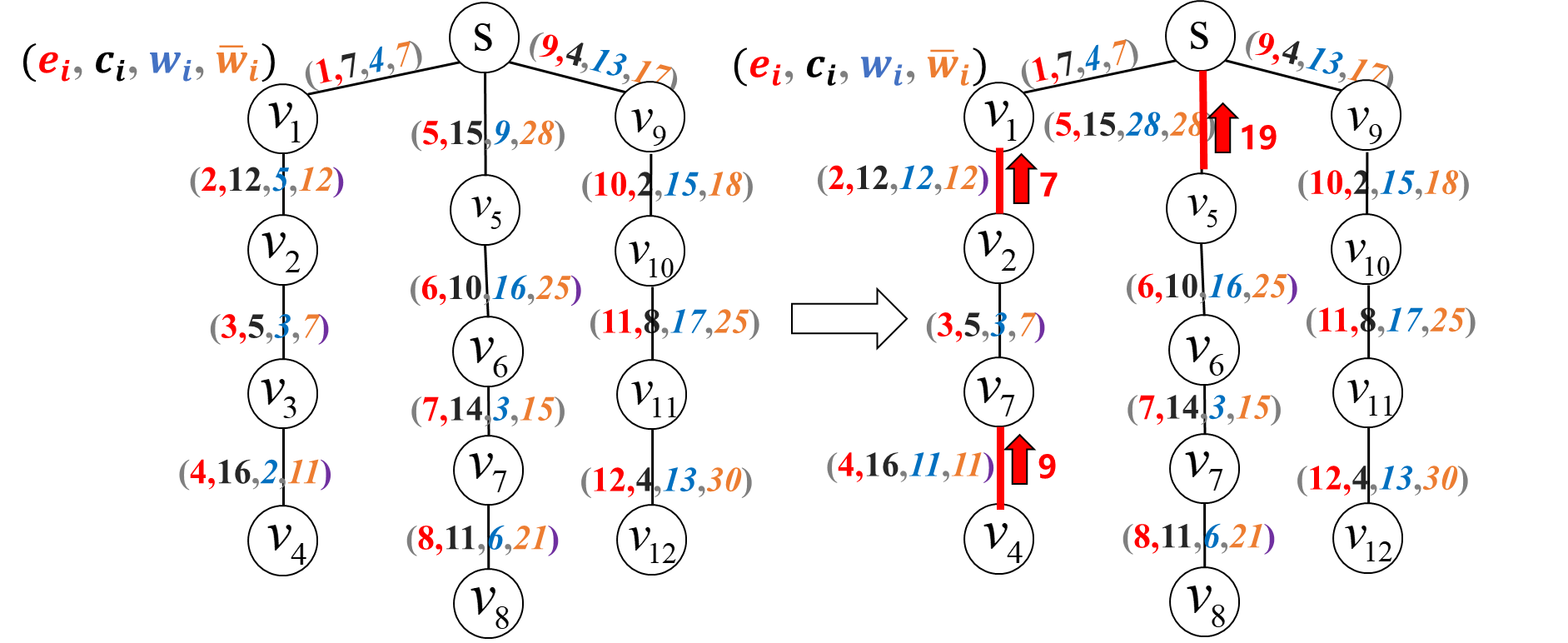}
    \caption{The left figure shows the weights $(e_i,c_i,w_i, \bar w_i)$. The right figure shows the upgraded edges in Alg. \ref{Alg_DITH_inf} when $\lambda=0.75$.}
    \label{example_3_fig}
\end{figure}
    
\end{example}

\subsection{Computational experiments}
We present the numerical experimental results for Algorithms \ref{Alg_DITH_inf_N1}, \ref{Alg_PIC_inf}, \ref{Alg_DITH_inf},
\ref{Alg_MCDITH_inf} in Table \ref{table:ex}.
These programs were coded in Matlab2023a and ran on an Intel(R) Core(TM) i7-10875H CPU @ 2.30GHz and 2.30 GHz machine running Windows 11. We tested these algorithms on six randomly generated trees with vertex numbers ranging from 10 to 500.
We randomly generated the vectors $u$, $c$ and $w$ such that $0\leq w \leq u$ and $c > 0$.
We generated $K$, $D$  and
$N$ for each tree based on $n$ with randomness, respectively.
In this table the average, maximum and minimum CPU time are denoted by $T_i$, $T_i^{max}$ and $T_i^{min}$, respectively, where $i = {1, \cdots, 4}$ represents Algorithms \ref{Alg_DITH_inf_N1}, \ref{Alg_PIC_inf}, \ref{Alg_DITH_inf},
\ref{Alg_MCDITH_inf}, respectively. 

\begin{table}[!htbp] 
	\centering
	\caption{Performance of Algorithms}
	\label{table:ex} 
	\begin{tabular}{ccccccc}\hline 
		
		\renewcommand{\arraystretch}{1.5pt}
		\renewcommand{\tabcolsep}{5pt}
		Complexity& $n$ &10 &50 &100 &300 &500 \\ 
            \hline
		
		$O(n)$ & $T_1$ &0.0002 & 0.0009 &0.00013 &0.0040 &0.0103 \\
		&  $T_1^{max}$ & 0.0004 &0.0015 &0.0034 &0.0324 &0.0177 \\
		&  $T_1^{min}$ & 0.0001 &0.0004 & 0.0010 &0.0027 &0.0051 \\
			\hline	
            $O(nN^2)$ & $T_2$ &0.0008 &0.0205 &0.0834 &0.7281 &2.3140 \\
            & $T_2^{max}$ &0.0020 &0.0274 &0.1521 &0.9820 &2.7254 \\
            & $T_2^{min}$ &0.0005 &0.0170 &0.0654 &0.6471 &1.7850 \\		
		\hline
            \(O(nN^2 \log_2 U)\) & $T_3$ &0.0010 &0.0531 &0.2487 &3.1651 &10.9810 \\
            & $T_3^{max}$ &0.0022 &0.0630 &0.2750 &5.6050 &14.0052 \\
            & $T_3^{min}$ &0.0005 &0.0420 &0.1630 &2.1240 &8.5100 \\

            \hline
            $O(nN^2 \log_2 K\log_2 U)$ & $T_4$ &0.0046 &0.2266 &0.9446 &15.5903 &48.5955 \\
            & $T_4^{max}$ &0.0110 &0.3234 &1.2330 &18.4107 &63.2560 \\
            & $T_4^{min}$ &0.0027 &0.2116 &0.8632 &12.4260 &43.4700 \\
		\hline
	\end{tabular}
\end{table}

From Table \ref{table:ex}, we can see that Algorithm \ref{Alg_MCDITH_inf} is time-consuming due to the repeatedly calling of Algorithm  \ref{Alg_DITH_inf} in pseudo-polynomial  time and the uncertainty of its iteration number.

Overall, these algorithms are all very effective and follow their respective time complexities well. When $n$ is small, the time differences among the three algorithms are relatively small, but as $n$ increases, the differences between the algorithms become more pronounced.

\section{Conclusion and Future Research}

This paper delves into the intricate dynamics of the double interdiction problem on the sum of root-leaf distance on trees, with a primary focus on maximizing SRD through edge weight adjustments within cost limitations and minimum path length requirements. By establishing parallels with the 0-1 kapsack problem, it illustrates the $\mathcal{NP}$-hardness of problem (DIT$_{H\infty}$). In addressing scenarios where a single upgrade is permissible, a pragmatic greedy algorithm is proposed to mitigate complexity. For situations necessitating multiple upgrades, an pseudo-polynomial time dynamic programming algorithm is advocated, striking a delicate balance between shortest path considerations and the summation of root-leaf distances. Specifically, when the weighted sum type is replaced by a cardinality constraint, the algorithm becomes a polynomial time algorithm.

Moreover, the paper ventures into the realm of the related minimum cost problem (MCDIT$_{H\infty}$), demonstrating its $\mathcal{NP}$-hardness through a reduction from the 0-1 knapsack problem. Subsequently, it outlines a binary search methodology to tackle the minimum cost predicament, culminating in a series of numerical experiments that vividly showcase the efficacy of the algorithms presented.

For further research, a promising avenue lies in extending similar methodologies to interdiction problems concerning source-sink path lengths, maximum flow, and minimum spanning trees, employing diverse metrics and measurements across general graphs. Such endeavors hold the potential to deepen our understanding and broaden the applicability of interdiction strategies in various real-world contexts.

\vskip 0.5cm
{\small \textbf{Funding} The Funding was provided by National Natural Science Foundation of China (grant no: 11471073).

\vskip 0.3cm
\textbf{Data availability}
 Data sharing is not applicable to this article as our datasets were generated randomly.

\section*{Declarations}

\textbf{Competing interests} The authors declare that they have no competing interest.

}

\end{document}